\documentclass[12pt]{article}
\usepackage{graphicx}
\usepackage{amssymb}
\usepackage{amsmath}
\usepackage{cite}
\usepackage{enumerate}
\newtheorem{thm}{Theorem}[section]
\newtheorem{cor}{Corollary}[section]
\newtheorem{lem}{Lemma}[section]
\newtheorem{prop}{Proposition}[section]
\newtheorem{defn}{Definition}[section]
\newtheorem{rem}{Remark}[section]

\newcommand{\eh}{\hfill}\newlength{\sperr}

\newenvironment{proof}{{\settowidth{\sperr}{\bf\rm
Proof}%
\par\addvspace{0.3cm}\noindent\parbox[t]{1.3\sperr}
{\textit{ P\eh r\eh o\eh o\eh f\eh .}}%
}}{\nopagebreak\mbox{}
$\Box$\par\addvspace{0.3cm}}

\numberwithin{equation}{section}

\def\s{\sigma}

\title{Effective construction  of a class of positive operators
in Hilbert
space, which
 do not admit  triangular factorization }
\author{Lev Sakhnovich}
\date{}
\begin{document}
\maketitle

\thanks{99 Cove  ave., Milford, CT, 06461, USA\\
lsakhnovich@gmail.com}

 \begin{abstract} A class  of
 non-factorable positive operators is constructed. As a result, pure existence
 theorems in the well-known problems by Ringrose, Kadison and Singer
are substituted by concrete examples.
\end{abstract}



\section{Introduction}\label{sec1}
To introduce the main notions of the
triangular factorization (see  \cite{3, 5, 8, 14, 15, 20})
consider a Hilbert space $L^{2}(a,\, b)\, $
$(-\infty{\leq}a<b{\leq}\infty)$. The orthogonal projectors $P_{\xi}$ in $L^{2}(a,\, b)$
are defined by the relations
\[
\big(P_{\xi}f\big)(x)=f(x)  \, \, \mathrm{for} \, \,  a<x<\xi,  \,\,
\big(P_{\xi}f\big)(x)=0 \, \, \mathrm{for} \, \,  \xi <x<b
\, \, \big( f \in L^{2}(a,b)\big).
\]
Denote the identity operator by $I$.
\begin{defn}\label{Definition 1.1.}
A bounded operator $S_{-}$ on
 $L^{2}(a,b)$  is called lower triangular if for
every
 $\xi$ the relations \begin{equation}\label{1.1}
  S_{-}Q_{\xi}=Q_{\xi}S_{-}Q_{\xi} , \end{equation} where
  $Q_{\xi}=I-P_{\xi}$, are true. The operator  $S_{-}^{\ast}$ is called
  upper triangular.
\end{defn}

\begin{defn} \label{Definition 1.2.}
 A bounded, positive definite and invertible
operator $S$ on $L^{2}(a,b)$ is said to admit a left (right)
triangular factorization if it can be represented in the form
\begin{equation}\label{1.2}
S=S_{-}S_{-}^* \quad (S=S_{-}^*S_{-}),\end{equation} where $S_{-}$ and $S_{-}^{-1}$ are
bounded and lower triangular operators.
\end{defn}
Further, we often write factorization meaning a left triangular factorization.

In paper \cite{20} (see p. 291)  we formulated necessary and sufficient conditions
under which the positive definite operator $S$ admits  a triangular
factorization. The factorizing operator $S_{-}^{-1}$ was
constructed in the explicit form.
We proved that a wide class of  operators admits a triangular factorization \cite{20}.

 D. Larson proved  \cite{8} the \emph{existence}  of positive definite and invertible
 but non-factorable operators. In the present article we construct  \emph{concrete examples} of
 such operators. In particular, the following operator
 \begin{equation}\label{1.3}   Sf=f(x)-{\mu}\int_{0}^{\infty}\frac{\mathrm{sin}{\pi}(x-t)}{{\pi}(x-t)}f(t)dt,\quad
 f(x){\in}L^{2}(0,\infty),\quad 0<\mu<1
 \end{equation}
 is positive definite and invertible but non-factorable.
 Using  positive definite and invertible but non-factorable operators  we have managed to substitute pure  existence  theorems \cite{8} by concrete examples in the
 well-known problems posed by J.R. Ringrose \cite{13}, R.V. Kadison and I.M. Singer \cite{5}.
 We note that Kadison-Singer problem was posed independently by
 I. Gohberg and M.G. Krein \cite{4}.

 The non-factorable operator $S$, which is defined by formula \eqref{1.3}, is used in a number of theoretical and applied problems (in optics \cite{7},
random matrices \cite{23}, generalized stationary processes \cite{11,12}, and Bose gas theory \cite{10}). The results obtained in  the paper are interesting
 from this point of view too.
 \section{A special class of operators \\ and corresponding differential systems}\label{sec2}

In this section we consider operators $S$ of the form
 \begin{equation}\label{2.1}
 Sf=f(x)-{\mu}\int_{0}^{\infty}h(x-t)f(t)dt,\quad
 f(x){\in}L^{2}(0,\infty),
 \end{equation}
 where $\mu=\overline{\mu}$ and  $h(x)$ admits representation
 \begin{equation}\label{2.2}
   h(x)=\frac{1}{2\pi}\int_{-\infty}^{\infty}e^{ix\lambda}\rho(\lambda)d\lambda.
 \end{equation}
 We suppose that the function $\rho(\lambda)$  satisfies the following conditions

1. The function $\rho(\lambda)$ is real and bounded
\begin{align}\label{d0}&
|\rho(\lambda)|{\leq}U^{2},\quad U>0 \quad  (-\infty<\lambda<\infty).
\end{align}

2.
  $\rho(\lambda)=\rho(-\lambda){\in}L(-\infty,\infty)$.\\
  Hence,  the function $h(x)\quad
  (-\infty<x<\infty)$ is continuous and real. The corresponding operator
  \begin{equation}\label{2.3}
   Hf=\int_{0}^{\infty}h(x-t)f(t)dt
  \end{equation}
  is self-adjoint and bounded, where $||H||{\leq}U.$
  We introduce the operators
   \begin{equation}\label{2.4}   S_{\xi}f=f(x)-{\mu}\int_{0}^{\xi}h(x-t)f(t)dt,\quad
 f(x){\in}L^{2}(0,\xi),\quad 0<\xi<\infty.\end{equation}
 The following statement is true.

 \begin{prop}\label{Proposition 2.1.}
 {If $-1/U<\mu<1/U$, then the operator}
 $S_{\xi}$, which is defined by formula \eqref{2.4}, is positive definite, bounded
 and invertible.
 \end{prop}

 Hence, we have
 \begin{equation}\label{2.5}
   S_{\xi}^{-1}f=f(x)+\int_{0}^{\xi}R_{\xi}(x,t,\mu)f(t)dt.
   \end{equation}
 The function $R_{\xi}(x,t,\mu)$ is jointly continuous in $x,t,\xi , \mu$.
 M.G. Krein (see \cite{4}, Ch. IV, Section 7) proved that
 \begin{equation}\label{2.6}
  S_{b}^{-1}=(I+V_{+})(I+V_{-}),\quad
 0<b<\infty,
 \end{equation}
 where the operators $V_{+}$ and $V_{-}$ are defined in $L^{2}(0,b)$ by the
 relations
\begin{equation}\label{2.7}
 \big(V_{+}^{\ast}f\big)(x)=\big(V_{-}f\big)(x)=\int_{0}^{x}R_{x}(x,t,\mu)f(t)dt.
\end{equation}
 The Krein's formula \eqref{2.6} is true for the Fredholm class of operators. The operator $S_{b}$ belongs to this class. The kernel of the operator
$V_{-}$ does not depend of $b$. Hence, if the operator $S$ admits the factorization,
then formula \eqref{2.6} holds for the case $b=\infty$ too, i.e.
\begin{equation}\label{2.8}   S^{-1}=(I+V_{+})(I+V_{-}).
\end{equation}

\begin{rem}\label{Remark 2.1.}
Relation \eqref{2.8} also follows  from Theorem 2.1 in the  paper \cite{20}.
\end{rem}

Let us introduce the function
\begin{equation}\label{2.9}
  q_{1}(x)=1+\int_{0}^{x}R_{x}(x,t,\mu)dt.
\end{equation}
Using the relation  $R_{x}(x,t,\mu)=R_{x}(x-t,0,\mu)$ (see \cite{4}, formula (8.12)), we obtain
\begin{equation}\label{2.10}
  q_{1}(x)=1+\int_{0}^{x}R_{x}(u,0,\mu)du.
  \end{equation}
According to the
well-known Krein's formula (\cite{4}, Ch. IV, formulas (8.3) and (8.14)) we have
\begin{equation}\label{2.11}
 q_{1}(x)=\exp\left\{\int_{0}^{x}R_{t}(t,0,\mu)dt\right\} .
 \end{equation}
Together with $q_{1}(x)$ we shall consider the function
\begin{equation}\label{2.12}
  q_{2}(x)=M(x)+\int_{0}^{x}M(t)R_{x}(x,t,\mu)dt,
\end{equation}
where
\begin{equation}\label{2.13}
  M(x)=\frac{1}{2}-{\mu}\int_{0}^{x}h(s)ds.
  \end{equation}
The functions $q_{1}(x)$ and $q_{2}(x)$ generate the $2{\times}2$ differential system
\begin{equation}\label{2.14}
  \frac{dW}{dx}=izJH(x)W,\quad W(0,z)=I_{2}.
  \end{equation}
  Here $W(x,z)$ and $H(x)$
are   $2{\times}2$ matrix functions and $J$ is a $2{\times}2$ matrix:
\begin{equation}\label{2.15}
  H(x)=\left[\begin{array}{cc}
                       q_{2}^{2}(x) & 1/2 \\
                       1/2 & q_{1}^{2}(x)
                     \end{array}\right], \quad
                     J=\left[\begin{array}{cc}
                        0 & 1 \\
                        1 & 0
                      \end{array}\right].
\end{equation}
Note that according to  \cite{19} (see formulas (53) and (56) therein) we have:
\begin{equation}\label{2.16}
  q_{1}(x)q_{2}(x)=1/2.
  \end{equation}
It is easy to see that
\begin{equation}\label{2.17}
 JH(x)=T(x)PT^{-1}(x),
 \end{equation}
where
\begin{equation}\label{2.18}
  T(x)=\left[\begin{array}{cc}
                             q_{1}(x) & -q_{1}(x) \\
                             q_{2}(x) & q_{2}(x)
                           \end{array}\right], \quad
P=\left[\begin{array}{cc}
    1 & 0 \\
    0 & 0
  \end{array}\right].
\end{equation}
Consider the matrix function
\begin{equation}\label{2.19}   V(x,z)=e^{-ixz/2}T^{-1}(x)W(x,z)T(0).\end{equation}
Due to \eqref{2.14}-\eqref{2.19} we get
\begin{equation}\label{2.20}   \frac{dV}{dx}=(iz/2)jV+\Gamma(x)V, \quad V(0)=I_{2},\end{equation}
where
\begin{equation}\label{2.21}   \Gamma(x)=\left[\begin{array}{cc}
                             0 & B(x) \\
                             B(x)& 0
                           \end{array}\right], \quad
j=\left[\begin{array}{cc}
          1 & 0 \\
          0 & -1
        \end{array}\right],
\end{equation}
\begin{equation}\label{2.22}
  B(x)=\frac{q_{1}^{\prime}(x)}{q_{1}(x)}=R_{x}(x,0,\mu).
  \end{equation}
Let us introduce the functions
\begin{equation}\label{2.23}   \Phi_{n}(x,z)=v_{1n}(x,z)+v_{2n}(x,z)\quad (n=1,2),\end{equation}
\begin{equation}\label{2.24}   \Psi_{n}(x,z)=i[v_{1n}(x,z)-v_{2n}(x,z)] \quad (n=1,2),\end{equation}
where $v_{in}(x,z)$ are elements of the matrix function $V(x,z)$.
It follows from \eqref{2.20} that
\begin{align}\label{2.25} & \frac{d\Phi_{n}}{dx}=(z/2)\Psi_{n}-B(x)\Phi_{n},\quad \Phi_{1}(0,z)=\Phi_{2},(0,z)=1, \\
& \label{2.26}   \frac{d\Psi_{n}}{dx}=-(z/2)\Phi_{n}+B(x)\Psi_{n},
\quad \Psi_{1}(0,z)=-\Psi_{2}(0,z)=i .\end{align}
Consider again the differential system \eqref{2.14} and the  solution $W(x,z)$ of this system.
The element $w_{1,2}(\xi,z)$ of the matrix function $W(\xi,z)$
can be represented in the form  (see \cite{17}, p. 54, formula \eqref{2.5})
\begin{equation}\label{2.27}   w_{1,2}(\xi,z)=iz\left(\left( I-zA\right)^{-1}1,S_{\xi}^{-1}1\right)_{\xi},\end{equation}
where the operator $A$ has the form
\begin{equation}\label{2.28}   Af=i\int_{0}^{x}f(t)dt.\end{equation}
It is well-known that
\begin{equation}\label{2.29}   (I-zA)^{-1}1=e^{izx}.\end{equation}
We can obtain a representation of $W(\xi,z)$ without using the operator $S_{\xi}^{-1}$.
Indeed, it follows from
\eqref{2.19},  \eqref{2.23}, and \eqref{2.24} that
\begin{equation}\label{2.30}   W(x,z)=(1/2)e^{ixz/2}T(x)\left[\begin{array}{cc}
                                            \Phi_{1}-i\Psi_{1}& \Phi_{2}-i\Psi_{2} \\
                                         \Phi_{1}+i\Psi_{1} & \Phi_{2}+i\Psi_{2}
                                          \end{array}\right]T^{-1}(0).\end{equation}
According to equality \eqref{2.10} we have $q_{1}(0)=1$. Due to \eqref{2.18} we infer
\begin{equation}\label{2.31}   T(0)=\left[\begin{array}{cc}
                             1& -1 \\
                             1/2 & 1/2
                           \end{array}\right],\quad T^{-1}(0)=\left[\begin{array}{cc}
                             1/2& 1 \\
                             -1/2 & 1
                           \end{array}\right].\end{equation}
Further we plan to use a Krein's result from \cite{6}. For that purpose we introduce the functions
\begin{equation}\label{2.32}
P(x,z)=e^{ixz/2}[\Phi(x,z)-i\Psi(x,z)]/2,
\end{equation}
\begin{equation}\label{2.33}
P_{\ast}(x,z)=e^{ixz/2}[\Phi(x,z)+i\Psi(x,z)]/2,
\end{equation}
where
\begin{equation}\label{2.34}
\Phi(x,z)=\Phi_{1}(x,z)+\Phi_{2}(x,z),\quad \Psi(x,z)=\Psi_{1}(x,z)+\Psi_{2}(x,z).
\end{equation}
Using \eqref{2.25}, \eqref{2.26} and  \eqref{2.32},   \eqref{2.33} we see that the pair $P(x,z)$ and $P_{\ast}(x,z)$ is a solution
of the following Krein system
\begin{equation}\label{2.35}   \frac{dP}{dx}=izP-B(x)P_{\ast},\quad  \frac{dP_{\ast}}{dx}=-B(x)P,
\end{equation} where
\begin{equation}\label{2.36}   P(0,z)=P_{\ast}(0,z)=1.\end{equation}
It follows from \eqref{2.32} and \eqref{2.33} that
\begin{equation}\label{2.37}   P(x,z)-P_{\ast}(x,z)=-ie^{ixz/2}\Psi(x,z).\end{equation}


\section{Non-factorable positive definite operators, a sufficient condition}\label{sec3}

We assume that the following relation  is true:
\begin{equation}\label{3.1}   M(x)=(1-\mu)/2+q(x),\quad q(x){\in}L^{2}(0,\infty),\end{equation}
where the function $M(x)$ is defined by \eqref{2.13}.
Condition \eqref{3.1} can be rewritten in an equivalent form:
\begin{equation}\label{3.2}   \int_{0}^{\infty}h(x)dx=1/2,\quad \int_{x}^{\infty}h(x)dx{\in}L^{2}(0,\infty).
\end{equation}
Now, we need the relations (see \cite{16}, Ch. 1, formulas (1.37) and (1.44)):
\begin{equation}\label{3.3}   S_{\xi}1=M(x)+M(\xi-x),\quad S_{\xi}=U_{\xi}S_{\xi}U_{\xi}, \end{equation}
where $U_{\xi}f(x)=\overline{f(\xi-x)},\quad 0{\leq}x{\leq}\xi.$
  It follows from \eqref{3.1} and \eqref{3.3} that
\begin{equation}\label{3.4}   S_{\xi}1=1-\mu +q(x)+U_{\xi}q(x).\end{equation}Hence the relation
\begin{equation}\label{3.5}   S_{\xi}^{-1}1=\frac{1}{(1-\mu)}[1-r_{\xi}(x)-U_{\xi}r_{\xi}(x)]\end{equation}
is true. Here  $r_{\xi}(x)=S_{\xi}^{-1}q(x).$ Using formulas \eqref{2.27},  \eqref{3.1}, and \eqref{3.5},
we obtain the following representation of  $w_{1,2}(\xi,z)$.
\begin{lem}\label{Lemma 3.1.}
{The function $w_{1,2}(\xi,z)$ has the form}
\begin{equation}\label{3.6}   w_{1,2}(\xi,z)=e^{iz\xi}G(\xi,z)-\overline{G(\xi,\overline{z})},\end{equation}
{where}
\begin{equation}\label{3.7}    G(\xi,z)=\frac{1}{1-\mu}\left[ 1-iz\int_{0}^{\xi}e^{-izx}r_{\xi}(x)dx\right].\end{equation}
\end{lem}
Note that the operator $S$ is positive definite, bounded and invertible.
 According to \eqref{2.6} we have
\begin{equation}\label{3.8}   Q(x)=(I+V_{-})q(x){\in}L^{2}(0,\infty).\end{equation}
Hence, there exists a sequence $x_{n}$ such that
\begin{equation}\label{3.9}   Q(x_{n}){\to}0,\quad x_{n}{\to}\infty.\end{equation}
Now, we prove the following statement.
\begin{lem}\label{Lemma 3.2.}
{Let relation \eqref{3.9} be true. Then we have}
\begin{equation}\label{3.10}
  \lim_{x_{n}{\to}\infty}{q_{1}(x_{n})}=\frac{1}{\sqrt{1-\mu}} .
\end{equation}
\end{lem}
\begin{proof}
In view of \eqref{2.9}, \eqref{2.12}, and \eqref{3.1} we get
\begin{equation}\label{3.11}
  q_{2}(x)=q_{1}(x)(1-\mu)/2+Q(x).
  \end{equation}
Taking into account the relation
 $q_{1}(x)q_{2}(x)=1/2$ (see \cite{19}, formulas (53) and  (56)),  we obtain   the equality
\begin{equation}\label{3.12}   1/2=q_{1}^{2}(x)(1-\mu)/2+q_{1}(x)Q(x).\end{equation}
 Formula  \eqref{3.10} follows directly from \eqref{3.9}, \eqref{3.12}, and inequality $q_{1}(x)>0$.
 \end{proof}

It follows from \eqref{2.18} and \eqref{3.10} that
 \begin{equation}\label{3.15}   T(x_{n}){\to}\left[\begin{array}{cc}
                            C &-C \\
                            1/2C & 1/2C
                          \end{array}\right],\quad x_{n}{\to}\infty,\quad C=1/\sqrt{(1-\mu)}.
                          \end{equation}
Hence, in view of \eqref{2.31}, \eqref{2.32}, \eqref{2.34}, and \eqref{3.15} the following assertion is true.
\begin{lem}\label{Lemma 3.3.} Let $x_n$ tend to $\infty$. Then,
$w_{1,2}$ has the following asymptotics
\begin{equation}\label{3.16}   w_{1,2}(x_{n},z)=-iCe^{ix_{n}z/2}\Psi(x_{n},z)\big(1+o(1)\big). 
\end{equation}
\end{lem}

\begin{lem} \label{Lemma 3.4.}
Suppose that the operator $S$ admits a factorization.
Then we have
\begin{equation}\label{3.17}    \lim_{\xi{\to}\infty}e^{-iz\xi}w_{1,2}(\xi,z)=G(z),\quad {\Im}z<0,\end{equation}
\begin{equation}\label{3.18}    \lim_{\xi{\to}\infty}w_{1,2}(\xi,z)=-\overline{G(\overline{z})},\quad {\Im}z>0.\end{equation}
{where}
\begin{equation}\label{3.19}    G(z)=\frac{1}{1-\mu}[1-iz\int_{0}^{\infty}e^{-izx}r(x)dx], \quad
r(x)=S^{-1}q(x).\end{equation}
\end{lem}
\begin{proof}
According to \eqref{2.8} we have
$S_{-}^{-1}=I+V_{-}$, where $V_{-}$ is defined by \eqref{2.7}.  Hence,  the operator function
 $S_{\xi}^{-1}$ strongly converges to the operator  $S^{-1}$ when $\xi{\to}\infty$. Then
 the function $r_{\xi}(x)=S_{\xi}^{-1}q(x)$ strongly converges to $r(x)=S^{-1}q(x)$, when
 $\xi{\to}\infty$, and
 $r(x){\in}L^{2}(0,\infty)$. Using \eqref{3.6} and \eqref{3.7} we obtain relations \eqref{3.17} and \eqref{3.18}. The lemma is proved.
 \end{proof}

 From Lemma \ref{Lemma 3.4.} we derive the following important assertion.

 \begin{prop}\label{Proposition 3.1.}
 {If at least one of the equalities \eqref{3.17} and  \eqref{3.18} is not true, then
the corresponding operator $S$ does not admit factorization.}
\end{prop}

Note that a new approach to the notion of the limit of a function was used in Lemma \ref{Lemma 3.2.}. Namely,
we introduce a continuous function $F(x)$, which belongs to $L(0,\infty)$, and consider
sequences $x_{n}{\to}\infty$, such that
\begin{equation}\label{d1} F(x_{n}){\to}0.\end{equation}
\begin{defn} \label{Definition 3.1.} We  say that the function $f(x)$ tends to $A$
almost sure $(a.s.)$ if relation \eqref{d1} implies
\begin{equation}\label{d2} f(x_{n}){\to}A,\qquad x_{n}{\to}\infty.\end{equation}
\end{defn}
Equality (3.10) can be written in the form
\begin{equation}\label{d3} \lim_{x{\to}\infty}{q_{1}(x)}=\frac{1}{\sqrt{1-\mu}},\quad  a.s. \end{equation}
\begin{rem}\label{Remark 3.1.}
From heuristic point of view "almost all" sequences $x_{n}{\to}\infty$
satisfy relation \eqref{d1}. This is the reason of using the probabilistic term
"almost sure".
\end{rem}

\section{A class of non-factorable positive definite operators}\label{sec4}

Introduce a partition
\begin{equation}\label{4.1}
0=a_{0}<a_{1}<...<a_{n}=a,
\end{equation}
and consider the function
$\rho(\lambda)=\rho(-\lambda)$ such that
\begin{equation}\label{4.2}
\rho(\lambda)= \left\{
\begin{array}{rl}
0, & a \leq \lambda, \\
b_{k-1}, & a_{k-1}{\leq}\lambda<a_{k},
\end{array}
\right.
\end{equation}
where
\begin{equation}\label{4.3}
b_{0}=1;\quad -1{\leq}b_{k}{\leq}1\quad (0<k{\leq}n-1).\end{equation}
In the case of $\rho$ given by \eqref{4.2} and  \eqref{4.3} we can put  $U=1$ in \eqref{d0}.
 Further we investigate the operators $S$, which are defined by formulas \eqref{2.1}, \eqref{2.2}, and
\eqref{4.2}. The spectral function $\sigma(\lambda)$ of the corresponding system
\eqref{2.35} is absolutely continuous and such that (see \cite{6}):
\begin{equation}\label{4.4}   \sigma^{\prime}(\lambda)=[1-\mu\rho(\lambda)]/(2\pi).\end{equation}
\begin{rem}\label{Remark 4.1.}
The operators $S$, which are defined by formulas \eqref{2.1}, \eqref{2.2}, and
\eqref{4.2}, appear in the theory of generalized stationary processes of
white noise type (see \cite{11,12}). If $n=1$ and $a_{1}=\pi$,  then the corresponding operator $S$ has  the form \eqref{1.3}.
\end{rem}

It follows from \eqref{2.2} and \eqref{4.2} that
\begin{equation}\label{4.5}   h(x)=\frac{1}{\pi}\sum_{k=1}^{n}b_{k}\frac{{\sin}a_{k}x-{\sin}a_{k-1}x}{x}.
\end{equation}
According to \eqref{4.4} we have
\begin{equation}\label{4.6}   \int_{-\infty}^{\infty}\frac{{\log}\sigma^{\prime}(u)}{1+u^{2}}du<\infty.
\end{equation}
It follows from \eqref{4.7}  (see \cite{6}) that
\begin{equation}\label{4.7}   \int_{0}^{\infty}|P(x,z_{0})|^{2}dx<\infty,\quad {\Im}z_{0}>0.\end{equation}
Hence, there exists a sequence $x_{n}$ such that
\begin{equation}\label{4.8}   |P(x_{n},z_{0})|^{2}{\to}0,\quad x_{n}{\to}\infty.\end{equation}
Now, we use the  corrected form of Krein's theorem (see \cite{6,21}):

\begin{prop}\label{Proposition 4.1.}
$1)$ { There exists the limit}
\begin{equation}\label{4.9}   \Pi(z)=\lim_{x_{n}{\to}\infty} P_{\ast}(x_{n},z),\end{equation}
{where the convergence is uniform at any bounded closed set  of the upper half-plane
$\Im z>0$. }\\
$2)$ {The function $\Pi(z)$ can be represented in the form}
\begin{equation}\label{4.10}    \Pi(z)=\frac{1}{\sqrt{2\pi}}\exp\left\{\frac{1}{2i\pi}\int_{-\infty}^{\infty}
\frac{1+tz}{(z-t)(1+t^{2})}(\log{\sigma}^{\prime}(t))dt+i\alpha\right\} ,\end{equation}
{where} $\alpha=\overline{\alpha}$. Here  $\sigma$ is the spectral function of system \eqref{2.35},
which corresponds to $\rho$ given by \eqref{4.2} and \eqref{4.3}, that is, this $\s$
is defined by \eqref{4.4}.
\end{prop}
\begin{rem}\label{Remark 4.2.}
The function $|Q(x)|^{2}+|P(x,z_{0})|^{2}$ belongs to
the space $L(0,\infty)$. Hence, there exists a sequence $x_{n}$ such that
relations \eqref{3.9} and  \eqref{4.8} are true simultaneously.
\end{rem}

If \eqref{4.5} holds, then the following conditions are fulfilled:
\begin{equation}\label{4.11}    0<\delta{\leq}||S||{\leq}\Delta<\infty,\qquad \int_{0}^{\infty}|h(x)|^{2}dx<\infty.\end{equation}
Therefore,
 in formula \eqref{4.10} we have (see \cite{19}, Proposition 1):
\begin{equation}\label{4.12}    \alpha =0.\end{equation}
One can easily  see that
\begin{equation}\label{4.13}   -\frac{1}{2i\pi}\int_{-\infty}^{\infty}
\frac{1+tz}{(z-t)(1+t^{2})}\log(2\pi)dt=\frac{1}{2}\log(2\pi). \end{equation}
It follows from \eqref{4.10}, \eqref{4.12}, and \eqref{4.13} that $\Pi(z)$ has the form
\begin{equation}\label{4.14}   \Pi(z)=
\prod_{k=0}^{n-1}\left[\left(\frac{a_{k+1}+z}{a_{k+1}-z}\right)\left(\frac{a_{k}-z}{a_{k}+z}\right)\right]^{\log(1-b_{k}\mu)/2i\pi},\quad {\Im}z>0.
\end{equation}
Next, we prove  the main result of this paper.
 \begin{thm}\label{Theorem 4.1.}
 { The bounded  positive definite and invertible operator $S$, which is defined by formulas \eqref{2.1} and \eqref{4.5},
does not admit a left triangular factorization.}
\end{thm}
\begin{proof}
Taking into account  Lemma 3.3 and relations  \eqref{2.37},  \eqref{4.8},  and \eqref{4.9} we have
 \begin{equation}\label{4.24}    \lim_{x_{n}{\to}\infty}w_{1,2}(x_{n},z)=-C\Pi(z),\quad {\Im}z>0,\quad C=1/\sqrt{(1-\mu)}.\end{equation}
Now, we use the following relations
\begin{equation}\label{4.25}   \lim_{y{\to}+0}\left(\frac{a_{k+1}-iy}{a_{k+1}+iy}\right)\left(\frac{a_{k}+iy}{a_{k}-iy}\right)=1,\quad k>0,
 \end{equation}
 \begin{equation}\label{4.26}   \lim_{y{\to}+0}\left(\frac{a_{k+1}-iy}{a_{k+1}+iy}\right)\left(\frac{a_{k}+iy}{a_{k}-iy}\right)=-1,\quad k=0.
 \end{equation}
 Formulas \eqref{4.14},  \eqref{4.25}, and \eqref{4.26} imply that
  \begin{equation}\label{4.27}   \lim_{y{\to}+0}\Pi(iy)=\sqrt{(1-\mu)}.\end{equation}

 Suppose that the operator $S$ admits a factorization.
It follows from  the asymptotics of sinus integral (see \cite{2}, Ch. 9, formulas (2) and (10)),
that the kernel $h(x)$, defined by formula \eqref{4.5}, satisfies conditions \eqref{3.2}.
Hence, the conditions of Lemma 3.4 are fulfilled.
 Comparing formulas \eqref{3.18} and \eqref{4.24},  we see that
\begin{equation}\label{4.29}   -\lim_{y{\to}+0}\overline{G(-iy)}=-1/(1-\mu){\ne}-C\lim_{y{\to}+0}\Pi(iy)=-1.\end{equation}
Hence, the relation
 \eqref{3.18} is not true. According to Proposition 3.1 the operator $S$ does not admit a factorization.
 The theorem is proved.
 \end{proof}


\section{Examples instead of existence theorems} \label{sec5}

Let the nest $N$ be the family of subspaces $Q_{\xi}L^{2}(0,\infty).$
The corresponding \emph{nest algebra} $Alg(N)$ is the algebra of all linear bounded operators
 in the space $L^{2}(0,\infty)$ for which every subspace from $N$ is an invariant subspace. Put $D_{N}=Alg(N){\bigcap}Alg(N)^{\ast}$.
 The set $N$ has \emph{ multiplicity one} if the diagonal $D_{N}$ is abelian, that is,
 $D_{N}$ is a commutative algebra.
 We can see that the lower triangular operators $S_{-}$ form the algebra
 $Alg(N)$, the corresponding diagonal $D_{N}$ is abelian, and it consists
 of the commutative  operators
 \begin{equation}\label{5.1}   T_{\varphi}f=\varphi(x)f,\quad f{\in}L^{2}(0,\infty),\end{equation}
 where $\varphi(x)$ is bounded. Hence, the introduced nest $N$ has the multiplicity $1$.\\

\noindent {\textbf{Ringrose Problem.}}
 \textit{Let} $N$ \textit{be a multiplicity one nest and} $T$ \textit{be a bounded invertible operator. Is} $TN$ \textit{necessarily multiplicity one nest?}\\

We obtain a concrete counterexample to Ringrose's hypothesis.

  \begin{prop}\label{Proposition 5.1.}
{Let the positive definite,
invertible operator $S$ be defined by the relations \eqref{2.1} and \eqref{4.5}.
 The  set $S^{1/2}N$ fails to have  multiplicity 1.}
 \end{prop}

 \begin{proof}
We use the well-known result (see \cite{3}, p. 169):\\
\emph{The following assertions are equivalent:}\\
\emph{1. The  positive definite, invertible operator T admits factorization.\\
2. $T^{1/2}$ preserves  the multiplicity of $N$.}\\
We stress that in our case the set $N=Q_{\xi}L^{2}(0,\infty)$ is fixed.)
The operator $S$  does not admit the factorization. Therefore,
the  set $S^{1/2}N$ fails to have  multiplicity $1$. The proposition is proved.
\end{proof}
Next, consider the operator
\begin{equation}\label{5.2}    Vf=\int_{0}^{x}e^{-(x+y)}f(y)dy, \quad f(x){\in}L^{2}(0,\infty).\end{equation}
An operator is said to be \emph{hyperintransitive} if its lattice of invariant
subspaces contains a multiplicity one nest. Note that the lattice of invariant
subspaces of the operator  $V$ coincides with $N$, see \cite{9} and \cite{22} (Ch. 11, Theorem 150). Hence we  deduce the answer
to Kadison-Singer \cite{5} and to Gohberg-Krein \cite{4} question.

\begin{cor} \label{Corollary 5.1.}
{The operator $W=S^{1/2}VS^{-1/2}$  is a non-hyperintransitive
compact operator.}
\end{cor}

Indeed, the lattice of the invariant subspaces of
 the operator $W$ coincides with $S^{1/2}N.$

\begin{rem}\label{Remark 5.1.}
The existence parts of Theorem \ref{Theorem 4.1.}, Proposition \ref{Proposition 5.1.}, and 
Corollary \ref{Corollary 5.1.}
 are proved by
D.R. Larson \cite{8}.
\end{rem}

\end{document}